\documentclass[journal,twoside,web]{ieeecolor}
\usepackage{generic}
\usepackage{cite}
\usepackage{amsmath,amssymb,amsfonts}
\usepackage{algorithmic,algorithm}
\usepackage{graphicx}
\usepackage{textcomp}
\usepackage{mathrsfs}  
\usepackage{subfigure}

\def\BibTeX{{\rm B\kern-.05em{\sc i\kern-.025em b}\kern-.08em
		T\kern-.1667em\lower.7ex\hbox{E}\kern-.125emX}}
\newtheorem{theorem}{Theorem}
\newtheorem{remark}{Remark}
\newtheorem{lemma}{Lemma}

\newtheorem{example}{Example}

\usepackage{etoolbox}
\makeatletter 
\pretocmd\@bibitem{\color{black}\csname keycolor#1\endcsname}{}{\fail}
\newcommand\citecolor[1]{\@namedef{keycolor#1}{\color{black}}}
\makeatother
\citecolor{coppersmith1982asymptotic}
\citecolor{bartels1972solution}
\citecolor{geromel1982optimal}
\citecolor{vidyasagar2002nonlinear}
\citecolor{dong2013robust}
\citecolor{geromel1979algorithm}

\begin{document}

\title{Second-Order Non-Convex Optimization for Constrained Fixed-Structure Static Output Feedback Controller Synthesis}

\author{
	Zilong Cheng,
	Jun Ma, 
	Xiaocong Li, 
	Masayoshi Tomizuka, \IEEEmembership{Life Fellow,~IEEE,}
	and Tong Heng Lee	
	\thanks{Z. Cheng and T. H. Lee are with the Integrative Sciences and Engineering Programme, NUS Graduate School, National University of Singapore, Singapore 119077 (e-mail: zilongcheng@u.nus.edu; eleleeth@nus.edu.sg).}
	\thanks{Jun Ma is with the Robotics and Autonomous Systems Thrust, The Hong Kong University of Science and Technology (Guangzhou), Guangzhou, China, and the Department of Electronic and Computer Engineering, The Hong Kong University of Science and Technology, Hong Kong SAR, China (e-mail: jun.ma@ust.hk).}
	\thanks{X. Li is with the John A. Paulson School of Engineering and Applied Sciences, Harvard University, Cambridge, MA 02138 USA (e-mail: xiaocongli@seas.harvard.edu).}
	\thanks{M. Tomizuka is with the Department of Mechanical Engineering, University of California, Berkeley, CA 94720 USA (e-mail: tomizuka@berkeley.edu).}
\thanks{This work has been submitted to the IEEE for possible publication. Copyright may be transferred without notice,	after which this version may no longer be accessible.}}

\maketitle

\begin{abstract}	
	
	For linear time-invariant (LTI) systems, 
	the design of an optimal controller is a commonly encountered problem in many applications. 
	Among all the optimization approaches available, 
	the linear quadratic regulator (LQR) methodology certainly garners much attention and interest.
	As is well-known, 
	standard numerical tools in linear algebra are readily available to determine
	the optimal static LQR feedback gain matrix when all the system state variables are measurable. 
	However, in various certain scenarios where some of the 
	system state variables
	are not measurable, 
	and consequent prescribed structural constraints on the controller structure arise, 
	the optimization problem 
	can become intractable 
	due to the non-convexity
	characteristics that can then be present. 
	In such cases, some first-order methods have been proposed to cater to these problems, 
	but all of these methods, if at all successful,
	are limited to linear convergence. 
	To speed up the convergence, 
	a second-order approach in the matrix space 
	is essential, with appropriate methodology
	to solve the linear equality constrained static output feedback (SOF) problem 
	with a suitably defined linear quadratic cost function.  
	Thus in this work, an efficient method 
	is proposed in the matrix space to calculate the Hessian matrix by solving several Lyapunov equations. 
	Then a new optimization technique is applied to deal with the indefiniteness of the Hessian matrix. 
	Subsequently, through Newton's method with linear equality constraints, 
	a second-order optimization algorithm is developed to 
	effectively solve the constrained SOF LQR problem. 
	Finally, two numerical examples are 
	described which demonstrate the applicability and effectiveness of the proposed method.

\end{abstract}


\section{Introduction}

Optimal control methodology
has certainly caught notable attention 
both in applications, and also in research and development efforts,
over the many recent past years~\cite{dolk2016output,jiang2019optimal}. 
During this time too, there has been
sustained technological advancement 
which has
facilitated and enabled 
the application of optimal control in 
quite a number of theoretical and practical problems~\cite{zhang2014distributed,wei2015value}.  
Among all the optimization approaches available, 
the linear quadratic regulator (LQR) methodology certainly garners much attention and interest~\cite{wu2018optimal,kanieski2015robust,zhang2015linear}.
In LQR problems, as is well-known, the cost function is defined to be a linear quadratic cost function 
in terms of the state variables and the control inputs; 
and the methodology is effective and straightforwardly applicable
when the dynamic system to be controlled 
can be modeled as linear and time-invariant.  
Here too, it needs to be noted that
the LQR methodology requires the availability of full state feedback as a prerequisite.  
However, in rather many practical applications, 
it can be a typical case that some of the system state variables are not measurable, nor available
for feedback purposes;
and such a situation can happen
arising possibly from 
real-world constraints of feasibility, complexity, and reconfigurability.  

When some of the system state variables are not measurable, 
certainly the alternate approach of
a static output feedback (SOF) controller can be utilized to satisfy 
the prescribed system performance requirements. 
With this approach, 
the optimal control problem can thus 
be formulated as the SOF LQR problem. 
A necessary condition for finding a stable solution 
for the SOF LQR problem are discussed in~\cite{levine1970determination}, 
and an iterative solution is obtained by solving the associated Lyapunov equations. 
Notably there, the controller gain resulting from the Lyapunov equations solution is a full matrix 
without any prescribed structural constraints. 
However, as indicated earlier, structural constraints in the controller gain can arise in certain scenarios; 
such as those, say, in decentralized control and sparse control problems. 
For these problems, it is then not straightforward to derive an optimal solution~\cite{ma2019parameter}. 
The evident reason here is that finding an optimal solution to the SOF problem 
is a Bilinear Matrix Inequality (BMI) optimization problem, 
which is generally non-convex~\cite{sadabadi2016static}. 
Moreover, it has been shown in~\cite{blondel1997np} that the SOF stabilization problem is an NP-hard problem; 
and unless it can be proved $P=NP$, there is no polynomial-time algorithm to solve this problem. 
In the existing literature then, most of the algorithms for finding 
a stable solution to the non-convex SOF problem are based on the Lyapunov equation approach, 
such as the D-K iteration optimization technique~\cite{el1994synthesis,lind1994evaluating}, 
the min-max iteration technique~\cite{geromel1994lmi,geromel1998static}, 
and the projection algorithm~\cite{peres1993h}. 
Also, a cone complementarity linearization algorithm proposed by~\cite{el1997cone} 
interestingly
introduces an efficient technique for finding a stable controller gain matrix 
with certain specifications.

To cater to the situation with structural constraints, 
substantial work actually has been conducted in the core area of gradient projection. 
In~\cite{ma2017integrated}, a first-order gradient projection method 
is implemented to enhance the linear quadratic performance; 
and which also considers the linear equality constraints 
such that the method can be used to solve decentralized control and sparse control problems. 
In~\cite{chanekar2017optimal}, generalized benders decomposition (GBD) 
and gradient projection are combined and utilized 
to solve a constrained linear quadratic problem on the condition 
that the closed-loop system is stable 
and a box constraint on the controller gain matrix is satisfied. 
However, all these existing algorithms utilize essentially the first-order method; 
and thus the convergence rate is limited. 
Here notably although not an unknown matter, 
yet due to the high complexity of calculating 
the Hessian matrix and the indefiniteness of the Hessian matrix, 
the more promising second-order methods are rarely used in 
developing effective solutions to
these non-convex optimal control problems. 
To the best of our knowledge, 
in available known developments,
approaches have been formulated where 
the Hessian matrix can only be calculated 
in terms of the entire controller gain matrix 
instead of separately element-wise~\cite{tassa2012synthesis,lin2013design}. 
Here when the controller gain matrix is sparse, or the dimension of 
the controller gain matrix is much less than the dimension 
of the system state,
the computational complexity of the Hessian matrix is then very high. 

With all of the above descriptions as a back-drop,
in this work here, we thus aim
to develop a second-order optimization approach to solve the SOF LQR problem effectively. 
An efficient method is proposed in the matrix space to calculate the Hessian matrix 
by solving several associated Lyapunov equations. 
Then a new optimization technique is applied to deal with the indefiniteness of the Hessian matrix. 
After that, through the constrained Newton's method, 
a second-order optimization method is developed to solve the 
specified constrained SOF LQR problem. 
To be more specific, in this work, based on a set of feasible solutions (which have been extensively reported in the literature), the sub-optimal points with respect to these feasible solutions can be efficiently and accurately found with our second-order optimization method. Subsequently, the sub-optimal point with the best performance index can be determined.
It is perhaps also worth mentioning and notable that 
the resulting proposed approach here is actually suitably generally applicable 
quite extensively to many various classes of commonly encountered optimal control problems, 
including the controller synthesis problem with prescribed sparsity pattern; the decentralized control problem; 
and certainly even the controller optimization problem without structural constraints.

The paper here is organized thus as follows: 
In Section~\ref{section:preliminary}, the constrained SOF LQR problem is elaborated;
and then the first-order method with gradient projection 
is also reviewed.
In Section~\ref{section:second_order}, 
we present and develop our second-order optimization method
where, firstly,
the Hessian matrix is derived 
with detailed discussions on dealing with the indefiniteness of the Hessian matrix. 
After that, the linear equality constrained Newton's method is given 
to solve the formulated optimization problem. 
In Section~\ref{section:numerical_example}, 
we consider the performance and effectiveness
of our proposed methodology on suitable 
illustrative examples,  
and the results here 
can certainly be seen 
to validate applicability and effectiveness of the proposed method. 
Section~\ref{section:conclusion} then concludes the paper with salient pertinent points.

\section{Preliminaries}\label{section:preliminary}

The following notations are used in the remaining text. $\mathbb R^{m\times n}$ ($\mathbb R^{n}$) denotes the real matrix with $m$ rows and $n$ columns ($n$ dimensional real column vector). $\mathbb S^{n}_{++}$ ($\mathbb S^{n}_{+}$) denotes the $n$ dimensional positive definite (positive semi-definite) real symmetric matrix. The symbol $A\succ0$ ($A\succeq0$) means that the matrix $A$ is positive definite (positive semi-definite). $A^T$ ($x^T$) denotes the transpose of the matrix $A$ (vector $x$). $J^{ij}$ denotes the single-entry matrix with a single entry $1$ located at the $i$th row and $j$th column, and the other entries are zero. $I$ represents the identity matrix with appropriate dimensions. The operator $\text{Tr}(\cdot)$ denotes the trace of a matrix. The operator $\langle \cdot,\cdot \rangle$ denotes the Frobenius inner product, i.e., $\langle A,B\rangle= \text{Tr}\left(A^TB\right)$ for $A,B \in \mathbb R^{m\times n}$. The norm operator based on the inner product operator is defined by $\|x\|_F=\sqrt{\langle x,x\rangle}$ for $x\in \mathbb R^{m\times n}$. The operator $\otimes$ denotes the kronecker product. The operator $\mathrm{vec}(\cdot)$ denotes the vectorization operator that expands a matrix by columns into a column vector. The operator $\text{det}(\cdot)$ denotes the determinant of a square matrix. $[A_1,A_2,\dots,A_n]$ ($[A_1;A_2;\dots;A_n]$) denotes the block matrix organized by rows (columns). $\mathbb E(\cdot)$ means the expectation. The operator $\lambda(\cdot)$ represents the eigenvalues of a matrix, and $\text{Re}(\cdot)$ returns the real part of a complex number. $\operatorname{diag}\{a_1, a_2, \ldots, a_n\}$ represents a diagonal matrix with numbers $a_i$, $\forall i = 1,2, \ldots, n$ as diagonal entries. $(\cdot)^*$ denotes the adjoint operator.

\subsection{Problem Statement}
A stabilizable and detectable linear time-invariant (LTI) system with an SOF controller can be expressed as
\begin{IEEEeqnarray}{rCl}\label{equation:linear_system}
\dot x(t)&=&Ax(t)+Bu(t)\IEEEyesnumber\IEEEyessubnumber\\
z(t)&=&C_1x(t)+D_1u(t)\IEEEyessubnumber\\
y(t)&=&Cx(t)\IEEEyessubnumber\\
u(t)&=&Ky(t),\IEEEyessubnumber
\end{IEEEeqnarray}
where $x(t) \in \mathbb R^n$ is the state vector,  $u(t)\in \mathbb R^{m}$ is the control input vector, $z(t)\in \mathbb R^p$ is the performance output vector used for specifying the system performance, $y(t)\in \mathbb R^q$ is the measured output vector for the controller, $A\in \mathbb R^{n\times n}$ is the state matrix, $B\in \mathbb R^{n\times m}$ is the input matrix, $C_1\in \mathbb R^{p\times n}$ and $D_1 \in \mathbb R^{p\times m}$ are the output matrix and the direct output matrix for specifying the system performance, $C\in \mathbb R^{q\times n}$ is the output matrix for the controller, and $K\in \mathbb R^{m\times q}$ is the SOF controller gain matrix. It is assumed that $C_1^TD_1=0$ and $D_1^TD_1\succ0$.

For an SOF linear quadratic optimization problem with respect to (\ref{equation:linear_system}), the cost function in the infinite horizon is defined as

\begin{IEEEeqnarray}{rCl}
J(K) &=&\color{black} \displaystyle\int_0^\infty z(t)^T z(t)dt\IEEEnonumber\\
&=&\color{black} \displaystyle\int_0^\infty \left[x(t)^TC_1^T  C_1 x(t)+u(t)^TD_1^T D_1 u(t)\right] dt.\IEEEeqnarraynumspace
\end{IEEEeqnarray} 
For simplicity, we define $\color{black} Q=C_1^T C_1$ and $\color{black} R=D_1^T D_1$ as the usual practice. Notably, $Q\in\mathbb S^n$ and $R\in\mathbb S^m$ must be ensured to be positive definite. Then the cost function can be converted to 
\begin{IEEEeqnarray}{rCl}\label{equation:cost_function}
J(K) &=& 
x_0^T\left(\displaystyle\int_0^\infty \Lambda_c(t)^T \left[Q+(KC)^T RKC\right] \Lambda_c(t)dt\right)x_0,\IEEEnonumber\\
\end{IEEEeqnarray} 
where $\Lambda_c(t) = e^{(A+BKC)t}$, and $x_0\in\mathbb R^n$ denotes the initial state vector of the system. The following matrices are used in the remaining text for the sake of brevity,
\begin{IEEEeqnarray}{rClr}
A_c&=&A+BKC\IEEEyesnumber\IEEEyessubnumber\label{equation:A_c}\\
Q_c&=&Q+(KC)^TRKC&\quad\IEEEyessubnumber\label{equation:Q_c}\\
X_0&=&x_0x_0^T\IEEEyessubnumber\\
P&=&\displaystyle\int _0^\infty \Lambda_c^T(t)Q_c\Lambda_c(t)dt.\IEEEyessubnumber
\end{IEEEeqnarray} 
Then the cost function can be expressed by
\begin{IEEEeqnarray}{rClr}
J(K)=\text{Tr} (PX_0).
\end{IEEEeqnarray} 

For the conventional LQR problem, the optimal solution is not affected by the initial condition. However, for the SOF LQR problem, due to the non-convexity of the optimization problem, a direct consequence is the initial condition dependence. In the literature, removing the initial condition dependence for this particular problem has most often been done by assuming that the initial condition is a stochastic uniformly distributed vector over the unit sphere and by minimizing the mathematical expectation of the quadratic cost~\cite{geromel1994decentralized}.

Define the set of the stable controller gains by $\mathscr K_s=\{K\in \mathbb R^{m\times q}\,|\, \max \{\text{Re} (\lambda(A_c))\}<0\}$. Then for each $K\in \mathscr K_s$, there exists a $P\in \mathbb S_{++}^{n}$ such that $A_c^TP+PA_c\prec 0$. Define the Lyapunov operator $L: \mathbb R^{n\times n}\rightarrow \mathbb R^{n\times n}$ given by $P\mapsto A_c^TP+PA_c$, where $A_c$ is defined in \eqref{equation:A_c}. 
\textcolor{black}{To describe 
the important properties of the Lyapunov operator
which is relevant at this point, 
the following important intermediate result, which is a restatement of Theorem 42 in~\cite{vidyasagar2002nonlinear}, is stated and highlighted here in the form of a lemma. }

\begin{lemma}\label{theorem:unique}\cite[Theorem 42]{vidyasagar2002nonlinear}
	For the LTI system~\eqref{equation:linear_system}, there exists a unique solution $P\in \mathbb S_{++}^{n}$ to the Lyapunov equation
	\begin{IEEEeqnarray}{rCl}\label{equation:lyapunov_equation}
		A_c^TP+PA_c+Q_c=0, \nonumber
	\end{IEEEeqnarray} 
	\textcolor{black}{where $A_c$ defined in~\eqref{equation:A_c} is Hurwitz and $Q_c$ defined in~\eqref{equation:Q_c} is positive definite.}
\end{lemma}
If there is no constraint on the LQR problem, and all system state variables
can be measured, then the optimal static state feedback gain can be directly obtained by solving the algebra Riccati equation (ARE). However, in some real-world applications, it is impossible to measure all of the system state variables. Moreover, some constraints on the controller structure must be considered. In these cases, the optimal controller gain matrix to the linear quadratic static state feedback problem cannot be directly obtained. In this work, we assume linear equality constraints are imposed on the controller structure, and we denote the controller parameters satisfying the desired linear equality constraints by $K\in {\mathscr C}$, where ${\mathscr C}=\{K\in {\mathbb R}^{m\times q}\,|\, {\mathcal C}(K)={\mathcal C}_0\}$. Considering the scenarios with multiple linear equality constraints, for all $i=1,2,\ldots,N$, we denote the linear equality constraints on the controller structure by
\begin{IEEEeqnarray}{rCl}\label{eqaution:matrix_constraints}
{\mathcal C}^{(i)}(K)&=&{\mathcal A}_1^{(i)}K{\mathcal B}_1^{(i)}+\dots+{\mathcal A}_{m_i}^{(i)}K{\mathcal B}_{m_i}^{(i)}={\mathcal C}_0^{(i)},
\end{IEEEeqnarray} 
where $\mathcal A_1^{(i)},\dots,\mathcal A_{m_i}^{(i)}$ and $\mathcal B_1^{(i)},\dots,\mathcal B_{m_i}^{(i)}$ are constraint matrices given by the optimization problem, $m_i$ in the subscript represents the number of constraint matrices in one equality for the $i$th equality constraint, and $N$ is the total number of the equality constraints. Then the constrained SOF problem can be summarized as
\begin{IEEEeqnarray}{rl}\label{eqn:optimization_problem}
\underset{K\in\mathbb R^{m\times q}}{\mathrm{minimize}}\quad&  J(K)\IEEEnonumber\\
\operatorname{subject\ to}\quad& \dot x(t)= Ax(t)+Bu(t)\IEEEnonumber\\
& u(t) = KCx(t)\IEEEnonumber\\
& K\in\mathscr C \cap \mathscr K_s.
\end{IEEEeqnarray}

Notably, the determination of an initial stabilizing controller with prescribed structure constraint is rather important, especially in the case when the system matrix is not stable (for instance, the second example in this paper). Actually this problem has been widely explored in the existing literature such as~\cite{crusius1999sufficient}.


\subsection{First-Order Optimization Method}
When the gradient projection method is applied to solve the constrained SOF problem, the problem can be divided into two sub-problems. Firstly, the gradient of the cost function with respect to the controller gain matrix without any constraint is obtained. Secondly, the unconstrained gradient is projected onto the linear equality constraints of the controller structure. By solving the two sub-problems in each iteration, we can obtain the descent direction of the linear quadratic cost function that preserves the linear equality constraints in the controller gain matrix. To solve the first sub-problem, Lemma~\ref{Lemma:gradient} is introduced.

\begin{lemma}\label{Lemma:gradient}
	\textcolor{black}{For the LTI system~\eqref{equation:linear_system} with the cost function~\eqref{equation:cost_function}, given that $Q_c$ defined in (4b) is positive definite and $X_0$ defined in (4c) is positive semi-definite,} the gradient of the cost function with respect to the controller gain matrix is given by
	\begin{IEEEeqnarray*}{c}\label{equation:gradient}
		\dfrac{dJ}{dK}=2\left(B^TP_g +RKC\right)\Gamma C^T,
	\end{IEEEeqnarray*}
	where $P_g\in\mathbb S^n_{++}$ and $\Gamma\in\mathbb S^n_+$ can be obtained by solving the following two Lyapunov equations,
	\begin{IEEEeqnarray*}{c}\label{equation:gradient_}
		LP_g =-Q_c,\quad L^*\Gamma =-X_0.
	\end{IEEEeqnarray*}
	
\end{lemma}

\begin{proof}
	\textcolor{black}{This lemma here above is a relatively straightforward generalization of a stated Lemma 1 in~\cite{geromel1979algorithm}. A full-length proof is available; but because of page limitation constraints, it is omitted here.
}
\end{proof}

After the gradient of the cost function with respect to the controller gain matrix is obtained, we can consider the linear equality constraints for the desired controller structure by using the gradient projection method.

\section{Second-Order Optimization Method}\label{section:second_order}

Essentially, gradient projection method can be used to solve the SOF LQR problem with the linear equality constraints. In most of the optimization problems, this method can work very well except for the slow convergence. One of the reasons is the linear convergence rate for most of the first-order optimization methods. Another reason is that the projection operation causes the loss of the gradient information. To deal with these drawbacks, we propose the following  optimization method.

\subsection{Derivation of the Hessian Matrix}

On the basis of Lemma~\ref{Lemma:gradient}, Theorem~\ref{theorem:hessian_matrix} is introduced to calculate the Hessian matrix of the cost function with respect to the controller gain matrix.

\begin{theorem}~\label{theorem:hessian_matrix}
	For the LTI system~(\ref{equation:linear_system}), the Hessian matrix of the cost function~(\ref{equation:cost_function}) with respect to the controller gain matrix can be expressed element-wisely by
	\begin{IEEEeqnarray}{rCl}\label{equation:hessian}
	\partial_K\partial_{k_{ij}} J &=& 2B^T\left[\left(P_1^{ij}\right)^T+P_1^{ij}\right]\Gamma C^T\IEEEnonumber\\
	&&+2\left[B^TP_g+RKC\right]\left[\left(\Gamma_1^{ij}\right)^T+\Gamma_1^{ij}\right]C^T\IEEEnonumber\\
	&&+2B^T\left[\left(R_1^{ij}\right)^T+R_1^{ij}\right]\Gamma C^T+2RJ^{ij} C \Gamma C ^T,\IEEEeqnarraynumspace
	\end{IEEEeqnarray}
	where $k_{ij}$ denotes the entry in the $i$th row and $j$th column of the gradient matrix, and a set of Lyapunov equations are defined as follows,
	\begin{IEEEeqnarray}{rCl}
	LP_1^{ij}&=&-P_gBJ^{ij}C\IEEEyesnumber\IEEEyessubnumber\label{equation:lypunov_hessian_1}\\
	L^*\Gamma_1^{ij}&=&-\Gamma \left(BJ^{ij}C\right)^T\IEEEyessubnumber\label{equation:lypunov_hessian_2}\\
	LR_1^{ij}&=&-(KC)^TRJ^{ij}C\IEEEyessubnumber\label{equation:lypunov_hessian_3}.
	\end{IEEEeqnarray}
\end{theorem}
\begin{proof}
	By Lemma~\ref{Lemma:gradient}, denote the gradient of the cost function in terms of the single element of the controller gain matrix in the inner product form, and then we have
	\begin{IEEEeqnarray}{rCl}
	\partial_{k_{ij}} J&=&2\left\langle{\Gamma ,P_gBJ^{ij}C}\right\rangle
	+2\left\langle{ \Gamma,(KC)^TR J^{ij}C}\right\rangle.
	\end{IEEEeqnarray}
	Then the Hessian matrix of the cost function can be expressed in the scalar form,
	\begin{IEEEeqnarray}{rCl}
	\partial_{k_{mn}}\partial_{k_{ij}} J
	&=&2\left\langle{\partial_{k_{mn}}\Gamma,P_gBJ^{ij}C}\right\rangle
	+2\left\langle{\Gamma ,\partial_{k_{mn}}P_gBJ^{ij}C}\right\rangle\IEEEnonumber\\
	&&+2\left\langle{\partial_{k_{mn}}\Gamma,(KC)^TR J^{ij}C}\right\rangle\IEEEnonumber\\
	&&+2\left\langle{\Gamma,\left(J^{mn}C\right)^TR J^{ij}C}\right\rangle.
	\end{IEEEeqnarray}
	Note that $\Gamma = -\left(L^*\right)^{-1}(X_0)$. Define $L_k^*\Gamma=\Gamma \left(\partial_kA_c^T\right)+\left(\partial_k A_c\right)\Gamma$,
	and then we have
	\begin{IEEEeqnarray}{rCl}
	\partial_{k_{mn}}\partial_{k_{ij}} J
	&=&2\left\langle{L^*_{k_{mn}}\Gamma,L^{-1}\left(-P_gBJ^{ij}C\right)}\right\rangle\IEEEnonumber\\
	&&+2\left\langle{\Gamma ,\left(-L^{-1}L_{k_{mn}}P_g-L^{-1}\partial_{k_{mn}}Q_c\right)BJ^{ij}C}\right\rangle\IEEEnonumber\\
	&&+2\left\langle{L^*_{k_{mn}}\Gamma,L^{-1}\left(-(KC)^TR J^{ij}C\right)}\right\rangle\IEEEnonumber\\
	&&+2\left\langle{\Gamma,\left(J^{mn}C\right)^TR J^{ij}C}\right\rangle.
	\end{IEEEeqnarray}
	Since we have $\left\langle{\Gamma ,\left(-L^{-1}L_{k_{mn}}P_g-L^{-1}\partial_{k_{mn}}Q_c\right)BJ^{ij}C}\right\rangle = \left\langle{\Gamma\left(BJ^{ij}C\right)^T ,-L^{-1}L_{k_{mn}}P_g-L^{-1}\partial_{k_{mn}}Q_c}\right\rangle$, the Hessian matrix is given by
	\begin{IEEEeqnarray}{RL}
	&\partial_{k_{mn}}\partial_{k_{ij}} J\IEEEnonumber\\
	=&2\left\langle{L^*_{k_{mn}}\Gamma,L^{-1}\left(-P_gBJ^{ij}C\right)}\right\rangle\IEEEnonumber\\
	&+2\left\langle{\left(L^*\right)^{-1}\left(-\Gamma\left(BJ^{ij}C\right)^T\right),L_{k_{mn}}P_g+\partial_{k_{mn}}Q_c}\right\rangle\IEEEnonumber\\
	&+2\left\langle{L^*_{k_{mn}}\Gamma,L^{-1}\left(-(KC)^TR J^{ij}C\right)}\right\rangle\IEEEnonumber\\
	&+2\left\langle{\Gamma,\left(J^{mn}C\right)^TR J^{ij}C}\right\rangle.
	\end{IEEEeqnarray}
	From $L_{k_{mn}}^*\Gamma=\Gamma \left(BJ^{mn}C\right)^T+\left(BJ^{mn}C\right)\Gamma$
	and \eqref{equation:lypunov_hessian_1}$-$\eqref{equation:lypunov_hessian_3}, it follows that
	\begin{IEEEeqnarray}{rCl}
	\partial_{k_{mn}}\partial_{k_{ij}} J
	&=&2\left\langle{\Gamma\left(BJ^{mn}C\right)^T+\left(BJ^{mn}C\right)\Gamma,P_1^{ij}}\right\rangle\IEEEnonumber\\
	&&+2\Big\langle\Gamma_1^{ij},\left(BJ^{mn}C\right)^TP_g+P_g\left(BJ^{mn}C\right)\IEEEnonumber\\
	&&\quad+\left(J^{mn}C\right)^TRKC+\left(KC\right)^TRJ^{mn}C\Big\rangle\IEEEnonumber\\
	&&+2\left\langle{\Gamma\left(BJ^{mn}C\right)^T+\left(BJ^{mn}C\right)\Gamma,R_1^{ij}}\right\rangle\IEEEnonumber\\
	&&+2\left\langle{\Gamma,\left(J^{mn}C\right)^TR J^{ij}C}\right\rangle.
	\end{IEEEeqnarray}
	Then the Hessian matrix in the trace form is expressed as
	\begin{IEEEeqnarray}{l}
	\quad\partial_{k_{mn}}\partial_{k_{ij}} J\IEEEnonumber\\
	=2\text{Tr}\left(C\Gamma P_1^{ij}BJ^{mn}\right)+2\text{Tr}\left(B^T P_1^{ij}\Gamma C^T\left(J^{mn}\right)^T\right)\IEEEnonumber\\
	+2\text{Tr}\left(C\Gamma_1^{ij}P_gBJ^{mn}\right)+2\text{Tr}\left(B^TP_g\Gamma_1^{ij} C^T\left(J^{mn}\right)^T\right)\IEEEnonumber\\
	+2\text{Tr}\left(C\Gamma_1^{ij}\left(KC\right)^TRJ^{mn}\right)+2\text{Tr}\left(RKC\Gamma_1^{ij}C^T\left(J^{mn}\right)^T\right) \IEEEnonumber\\
	+2\text{Tr}\left(C\Gamma R_1^{ij}BJ^{mn}\right)+2\text{Tr}\left(B^TR_1^{ij}\Gamma C^T\left(J^{mn}\right)^T\right)  \IEEEnonumber\\
	+2\text{Tr}\left(R J^{ij}C\Gamma C^T\left(J^{mn}\right)^T\right).
	\end{IEEEeqnarray}
	By the continuity and linearity of the trace operator, the Hessian matrix can be expressed as (\ref{equation:hessian}).
	This completes the proof of Theorem~\ref{theorem:hessian_matrix}.
\end{proof}

\subsection{Indefiniteness of the Hessian Matrix}

The indefiniteness of the Hessian matrix is a pervasive problem existing in the non-convex optimization problems. The algorithms on the second-order optimization for the nonlinear optimization problems have been widely studied.
Intuitively, finding a locally optimal point for the non-convex problem should be as simple as finding a globally optimal point for the non-convex problem, but in practice, the fact is that many more steps are required to achieve the locally optimal point. This is because of the pervasively existing saddle points in the non-convex problems. It has been shown that for the non-convex optimization problems, it is the saddle points that impede the optimization procedures~\cite{dauphin2014identifying}. Therefore, how to evade the saddle points becomes a critical problem. 

An intuitive solution to evade the saddle point is to rescale the gradient vector by the inverse of the absolute value of the corresponding eigenvalue, i.e., rescale $\left({dJ}/{dK}\right)_i$ by ${1}/{|\lambda_i|}$, where $\lambda_i$ is the $i$th eigenvalue of the Hessian matrix\cite{dauphin2014identifying}. Adding an identity matrix to the indefinite Hessian matrix such that the matrix $(\alpha I+H)$ is positive definite~\cite{tassa2012synthesis} and using the absolute value of the Hessian matrix~\cite{nocedal2006numerical} are also commonly used in the existing literature. However, there is no theoretical support for such techniques so far and even no intuitive explanation. 
Even though many algorithms have been proposed, how to evade the saddle point when the second-order methods are used for the non-convex optimization problems is still an open question. In this paper, the positive definite truncated (PT)-inverse method proposed by~\cite{paternain2019newton} is utilized.
Since the PT-inverse can guarantee that the Hessian matrix is positive definite, the iteration steps are in the proper descent direction. The sub-optimal point can be definitely achieved alongside this direction.  

\subsection{Equality Constrained Newton's Method}

Since the controller gain matrix $K\in \mathbb R^{m\times q}$ is not in a vector form, the Hessian matrix of the cost function cannot be denoted explicitly. By expanding the controller gain matrix into the vector form, we can do the optimization in terms of the vector form controller gain. After that, the controller gain can be easily converted to the matrix form for further implementation.

Theorem~\ref{theorem:vector_constraints} shows that the linear equality constraints can be expressed explicitly in the vector form.
\begin{theorem}\label{theorem:vector_constraints}
	The linear equality constraints defined in \eqref{eqaution:matrix_constraints} can be converted to the vector form, which can be expressed as
	\begin{IEEEeqnarray}{rCl}
	\bar{\mathcal{A}} \mathrm{vec}(K) = \bar{\mathcal{C}},
	\end{IEEEeqnarray}
	where
	\begin{IEEEeqnarray}{rCl}\label{equation:transformation}
	\bar{\mathcal{A}} &=& 
	\Bigg[
	\sum_{i=1}^{m_1}\left(\left({\mathcal B}_i^{(1)}\right)^T\otimes{\mathcal A}_i^{(1)}\right);
	\sum_{i=1}^{m_2}\left(\left({\mathcal B}_i^{(2)}\right)^T\otimes{\mathcal A}_i^{(2)}\right);\dots;\IEEEnonumber\\
	&&\sum_{i=1}^{m_N}\left(\left({\mathcal B}_i^{(N)}\right)^T\otimes{\mathcal A}_i^{(N)}\right)\Bigg]\IEEEyesnumber\IEEEyessubnumber\\
	\bar{\mathcal{C}} &=& \left[\mathrm{vec} \left({\mathcal C}_0^{(1)}\right); \mathrm{vec} \left({\mathcal C}_0^{(2)}\right);\dots;\mathrm{vec} \left({\mathcal C}_0^{(N)}\right)\right].\IEEEyessubnumber
	\end{IEEEeqnarray}
\end{theorem}

\begin{proof}
	By doing the vectorization in both sides to the constraints expressed in the matrix form as shown in (\ref{eqaution:matrix_constraints}), we can derive
	\begin{IEEEeqnarray}{rCl}
	\Bigg[\sum_{i=1}^{m_j}\left(\left({\mathcal B}_i^{(j)}\right)^T\otimes{\mathcal A}_i^{(j)}\right)\Bigg]\mathrm{vec}(K) = \mathrm{vec}\left({\mathcal C}_0^{(j)}\right),
	\end{IEEEeqnarray}
	where $j$ denotes the $j$th linear equality constraint. Then Theorem~\ref{theorem:vector_constraints} can be easily proved if all the equations are denoted in a block matrix form.
\end{proof}


For the linear equality constrained Newton's method, we need to ensure that the point after each iteration must stay in the feasible region, i.e., $\Bar{\mathcal{A}} \mathrm{vec} (K+\Delta K) = \Bar{\mathcal{C}}$. Therefore, if the stability constraint condition is ignored temporarily, we have the following optimization problem at a specific point $K=K_s$,
\begin{IEEEeqnarray}{L}\label{equation:Newton_optimization_problem}
\underset{\mathrm{vec}(\Delta K)\in\mathbb R^{mq}}{\mathrm{minimize}} \quad \bar J(\mathrm{vec}(K_s+\Delta K))\IEEEnonumber\\ 
\quad\quad\quad =J(K_s)+G_v^T\mathrm{vec}(\Delta K)+\frac{1}{2}\mathrm{vec}(\Delta K)^T H_v\mathrm{vec}(\Delta K) \IEEEnonumber\\
\;\;\text{subject to}  \quad\Bar{\mathcal{A}} \left(\mathrm{vec} (K_s)+\mathrm{vec}(\Delta K)\right) = \Bar{\mathcal{C}},
\end{IEEEeqnarray}

By using the analytical solution to the linear quadratic optimization problem, we can denote (\ref{equation:Newton_optimization_problem}) in the matrix form,
\begin{IEEEeqnarray}{LL}\label{equation:Newton_step}
\begin{bmatrix}
H_v & \bar {\mathcal{A}}^T\\
\bar {\mathcal {A}} & 0
\end{bmatrix}
\begin{bmatrix}
\mathrm{vec}(\Delta K)\\ w
\end{bmatrix}=
\begin{bmatrix}
-G_v\\0
\end{bmatrix},
\end{IEEEeqnarray}
where $w$ is the dual variable vector with the appropriate dimension for the linear quadratic optimization problem, $G_v\in\mathbb R^{mq}$ and $H_v\in \mathbb R^{mq\times mq}$ are given as
\begin{IEEEeqnarray}{rCl}
G_v&=&\mathrm{vec}\left(\frac{dJ}{dK}\right)\IEEEyesnumber\IEEEyessubnumber~\label{equation:gradient_vector}\\
H_v&=&\bigg[
\mathrm{vec}\left(\frac{\partial^2 J}{\partial k_{11}\partial K}\right),
\dots,
\mathrm{vec}\left(\frac{\partial^2 J}{\partial k_{m1}\partial K}\right),\IEEEnonumber\\
&&\quad\mathrm{vec}\left(\frac{\partial^2 J}{\partial k_{1q}\partial K}\right),
\dots,
\mathrm{vec}\left(\frac{\partial^2 J}{\partial k_{mq}\partial K}\right)
\bigg].\IEEEyessubnumber~\label{equation:hessian_vector}
\end{IEEEeqnarray}
Then in each iteration, we can derive the Newton step $\mathrm{vec}(\Delta K)$ by solving (\ref{equation:Newton_step}).

However, since this problem is non-convex, the indefiniteness of the Hessian matrix must be considered. Integrated with the PT-inverse method, the Newton step $\mathrm{vec}(\Delta K)$ is given by solving the following matrix equation,
\begin{IEEEeqnarray}{LL}\label{equation:Newton_step_PT}
\begin{bmatrix}
H_{v,\epsilon} & {\bar {\mathcal{A}}}^T\\
\bar {\mathcal {A}} & 0
\end{bmatrix}
\begin{bmatrix}
\mathrm{vec}(\Delta K)\\ w
\end{bmatrix}=
\begin{bmatrix}
-G_v\\0
\end{bmatrix},
\end{IEEEeqnarray}
where $H_{v,\epsilon}$ is the PT-matrix for the Hessian matrix $H_v$. To calculate the PT-matrix,
we use the singular value decomposition (SVD). Denote $H_v=M\Lambda M^T$, where $M\in\mathbb R^{n\times n}$ is a unitary matrix, and $\Lambda \in \mathbb S^n$ is a diagonal matrix. Define the positive definite truncated eigenvalue matrix $\Lambda_\epsilon$ with the parameter $\epsilon$ as
\begin{IEEEeqnarray}{c}
(\Lambda_\epsilon)_{ii}=
\begin{cases}
|\Lambda_{ii}|& \text{if } |\Lambda_{ii}|\ge \epsilon \\
\epsilon & \text{otherwise}.
\end{cases}
\end{IEEEeqnarray}
The PT-matrix of the Hessian matrix $H_v$ with the parameter $\epsilon$, which is denoted by $H_{v,\epsilon}$, is given by $H_{v,\epsilon} = M\Lambda_\epsilon M^T$.

From~\cite{nocedal2006numerical}, we can guarantee that each step $\mathrm{vec}(\Delta K)$ is a descent step. Since the cost function value of an unstable system is infinite, the stability of the system can be guaranteed if the cost function value belongs to a decreasing sequence as long as the initial gain stabilizes the closed-loop system.

Algorithm~\ref{algorithm: back_tracking_line_search} is introduced to summarize the modified backtracking line search used in this paper.
Then the linear equality constrained second-order non-convex optimization algorithm is summarized in Algorithm~\ref{algorithm:optimization}.
\begin{remark}
	$P_g$ must be positive definite to ensure the stability of the system. This can be easily seen from the Lyapunov stability theorem of the linear system. Therefore, for each step, the positive definiteness of the $P_g$ matrix must be guaranteed in the backtracking line search algorithm. \textcolor{black}{The convergence of the backtracking line search algorithm is guaranteed with the assumption of the existence of a stable solution~\cite{dong2013robust}.}
\end{remark}


\begin{algorithm}
	\caption{Backtracking line search with guaranteed stability }
	\begin{algorithmic}[1]\label{algorithm: back_tracking_line_search}
		\renewcommand{\algorithmicrequire}{\textbf{Input:}}
		\renewcommand{\algorithmicensure}{\textbf{Output:}}
		\REQUIRE Current controller gain matrix $K$, descent direction $\Delta K$, gradient $dJ/dK$, and backtracking parameters $\alpha \in (0,0.5)$, $\beta\in(0,1)$
		\ENSURE Controller gain matrix after iteration $K'$
		\\ \textit{Initialization} $t=1$:
		\WHILE {\TRUE}
		\STATE Compute $P_g\in\mathbb R^{n\times n}$ for $J(K+t\Delta K)$ \\
		\IF {$J(K+t\Delta K) < J(K)+\alpha t \text{Tr}\left(\left(dJ/dK\right)^T\Delta K\right)$ \AND $\min \{\text{eig}(P_g)\} > 0$}
		\STATE \textbf{break}
		\ELSE
		\STATE $t=\beta t$
		\ENDIF
		\ENDWHILE
		\RETURN $K'=K+t\Delta K$
	\end{algorithmic} 
\end{algorithm}

\begin{algorithm}
	\caption{Second-order optimization algorithm for the SOF LQR problem}
	\begin{algorithmic}[1]\label{algorithm:optimization}
		\renewcommand{\algorithmicrequire}{\textbf{Input:}}
		\renewcommand{\algorithmicensure}{\textbf{Output:}}
		\REQUIRE Stable controller gain matrix $K$, tolerance $\varepsilon>0$
		\ENSURE Sub-optimal controller gain matrix $K^*$
		\WHILE {\TRUE}
		\STATE Compute the gradient vector of the cost function $G_v$ by \eqref{equation:gradient} and (\ref{equation:gradient_vector})
		\STATE Compute the Hessian matrix of the cost function $H_v$ by \eqref{equation:hessian} and (\ref{equation:hessian_vector})
		\STATE Compute the PT-matrix $H_{v,\epsilon}$ of the Hessian matrix $H_v$
		\STATE Compute the Newton step $\mathrm{vec}(\Delta K)$ by (\ref{equation:Newton_step_PT})
		\IF {$\|\mathrm{vec}(\Delta K)\|\le \varepsilon$}
		\STATE \textbf{break}
		\ENDIF
		\STATE  Conduct the line search using Algorithm~\ref{algorithm: back_tracking_line_search} to find the controller gain matrix $K'$ for the next iteration
		\ENDWHILE
		\RETURN $K^*=K$
	\end{algorithmic} 
\end{algorithm}

In terms of the complexity for deriving the gradient, the Lyapunov equation can be solved with the Bartels-Stewart algorithm~\cite{bartels1972solution}, which has the complexity $\mathcal O(n^3)$. Because $C$ and $R$ are usually sparse matrices, the complexity of the matrix multiplication here is $\mathcal O(mn^2)$. Therefore, the complexity for deriving the gradient of the cost function is $\mathcal O(n^3+mn^2)$. The complexity for deriving the Hessian matrix is the same. However, finding the Newton direction involves the inverse operation of a matrix with the dimension $mq+k$, where $k$ denotes the row dimension of the equality constraint parameter matrix $\bar {\mathcal A}$. With the Coppersmith-Winograd algorithm~\cite{coppersmith1982asymptotic}, the complexity for deriving the Newton direction is  $\mathcal O((mq+k)^{2.373})$. The line search method is with the complexity $\mathcal O(n^3)$. Thus the complexity of the proposed algorithm is given by $\mathcal O((mq+k)^{2.373} + n^3+mn^2)$.
	
\section{Numerical Examples}\label{section:numerical_example}
In this section, two appropriate examples are worked through
to demonstrate the effectiveness and applicability 
of the proposed second-order optimization method here. 
The first example is a benchmark problem presented in~\cite{choi1974computation}, where the numerical data in this example refers to a Mach 2.7 flight condition of a supersonic transport aircraft. It aims to design an SOF controller for a given fourth-order system without any constraints. 
The second example focuses on the design of a linear equality constrained SOF controller 
for a third-order decentralized system. 
Both the first-order optimization algorithm with the gradient projection method and 
the proposed second-order optimization algorithm here are applied to solve the SOF problem. 
Comparative results are given to demonstrate the performance of both methods.  
Both of the SOF problems in the given examples are solved on 
a computer with 16G RAM and a 2.2GHz i7-8750H processor (6 cores), 
and the optimization algorithm is implemented and executed on MATLAB R2019b
(essentially a rather commonly available engineering development/computation environment presently).
\begin{example}
	The fourth-order system for an aircraft system is given by
	\begin{IEEEeqnarray}{rCl}
	\dot x(t) = Ax(t)+Bu(t) \quad y(t)= Cx(t)\quad u(t)=Ky(t),\IEEEeqnarraynumspace 
	\end{IEEEeqnarray}
	where
	\begin{IEEEeqnarray}{l}
	A=\left[
	\begin{array}{llll}
	-0.03700 & \phantom+0.01230 & \phantom+0.00055 & -1.00000\\
	\phantom+ 0.00000 & \phantom+0.00000 & \phantom+1.00000 & \phantom+0.00000\\
	-6.37000 & \phantom+0.00000 & -0.23000 & \phantom+0.06180\\
	\phantom+ 1.25000 & \phantom+0.00000 & \phantom+0.01600 & -0.04570
	\end{array}
	\right]\IEEEnonumber\\
	B=\left[
	\begin{array}{ll}
	\phantom+0.000840 & \phantom+0.000236\\
	\phantom+0.000000 & \phantom+0.000000\\
	\phantom+0.080000 & \phantom+0.804000\\
	-0.086200 & -0.066500
	\end{array}
	\right]
	C=\left[
	\begin{array}{llll}
	0 & 1 & 0 & 0\\
	0 & 0 & 1 & 0\\
	0 & 0 & 0 & 1
	\end{array}\right].\nonumber\\
	\end{IEEEeqnarray}
	
	An optimal controller $K$ is designed to minimize the cost function as given by
	\begin{IEEEeqnarray}{rCl}
	J &=&\displaystyle\int_0^\infty \left(x(t)^TQ x(t)+u(t)^TR u(t)\right) dt,
	\end{IEEEeqnarray}
	where the weighting parameters are chosen as $Q=I,\; R=I$ for demonstrative purposes. The system initial state matrix is chosen as a random vector with $\mathbb E\left(x_0x_0^T\right)=I$. 
\end{example}

The initial controller gain matrix is chosen as a zero matrix, with which the closed-loop system is stable. The stopping criterion is chosen as $\varepsilon = 10^{-9}$. For both of the first-order optimization method and the second-order optimization method, Algorithm~\ref{algorithm: back_tracking_line_search} is used to choose the suitable step size. The parameters for the backtracking line search are chosen as $\alpha = 0.2$ and $\beta = 0.1$. The parameter for the PT-matrix, which will be used in the second-order optimization method, is chosen as $\epsilon =  10^{-9}$.

\begin{figure}
	\centering
	\includegraphics[trim=100 50 120 50,width=1\columnwidth]{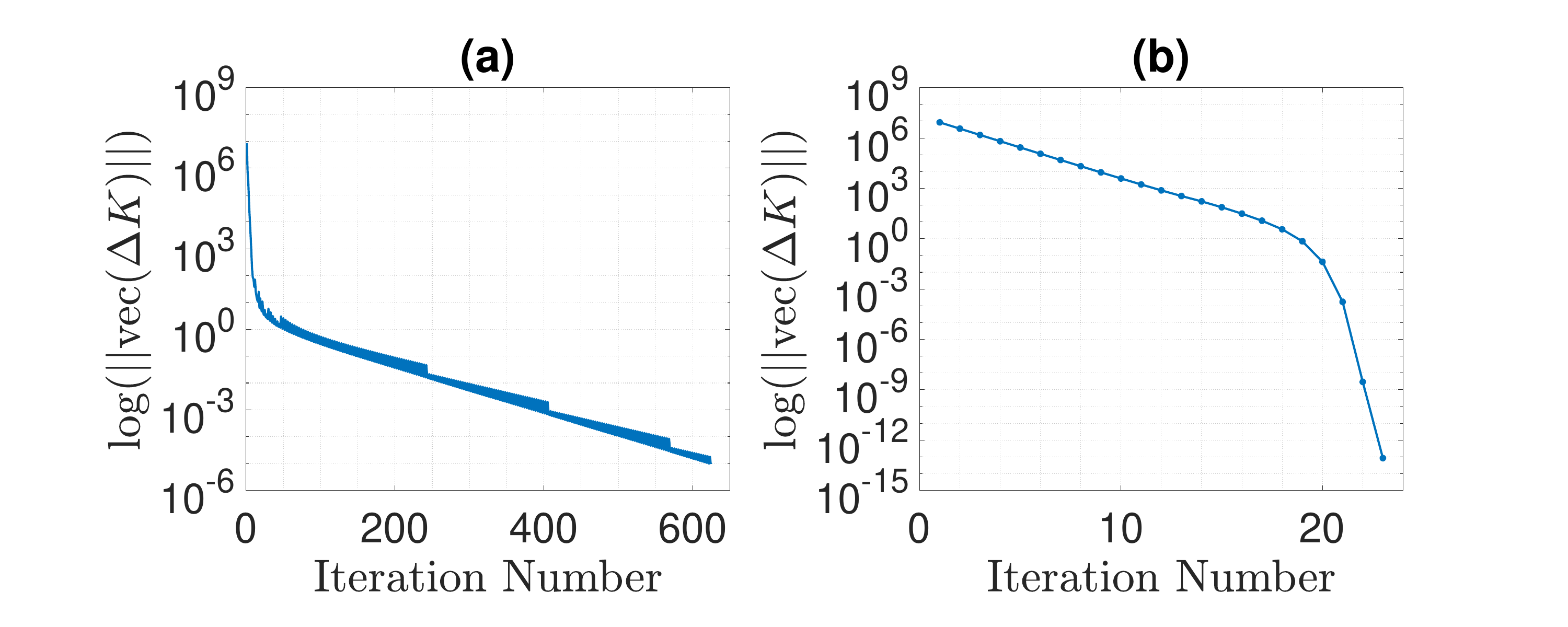}
	\caption{Norm of the gradient in Example 1. (a) First-order method. (b) Proposed method.}
	\label{fig:cost1}
\end{figure}

Fig.~\ref{fig:cost1}(a) shows the norm of the gradient of the cost function with respect to the controller gain matrix with the first-order optimization method. It can be seen that the norm has a decreasing trend after iterations with the first-order optimization method. Since it takes too many iterations to satisfy the stopping criterion, and the tendency for the curve of the norm of the gradient is much more clear with less data point, a relaxed stopping criterion $\varepsilon = 1\times 10^{-5}$ is chosen for the first-order method. It takes 624 iterations to achieve the sub-optimal point with the norm of the gradient $\|\mathrm{vec}(\Delta K^*)\|=9.5772\times 10^{-6}$. If the number of backtracking line search iterations is also taken into consideration, it takes in total 1696 iterations to reach the sub-optimal point with the defined stopping criterion. It can be seen that except for the very beginning iterations, the convergence rate is linear in most of the iterations.

The sub-optimal point with the first-order optimization method in this example is $J(K^*)=159.0686$. It takes 154.4332 seconds to reach this sub-optimal point. The sub-optimal parameter matrix given by the first-order method is
\begin{IEEEeqnarray}{l}
K^*_{(1)}=\left[
\begin{array}{llll}
\phantom+0.3975 & \phantom+1.5925 & \phantom+7.8522 \\
-1.2575 & -3.4823 & -5.0040 
\end{array}
\right].
\end{IEEEeqnarray}


Fig.~\ref{fig:cost1}(b) shows the norm of the gradient of the cost function with respect to the controller gain matrix with the proposed second-order optimization method. It can be seen that the norm decreases after each iteration with the second-order optimization method. Compared with the first-order optimization method, the second-order optimization method 
shows significantly higher convergence. 
It only takes 23 iterations to achieve the sub-optimal point with the norm of the gradient $\|\mathrm{vec}(\Delta K^*)\|=1.7571\times 10^{-13}$. If we consider the number of backtracking line search, it totally takes 24 iterations to reach the sub-optimal point with this norm. Therefore, in this example, the backtracking line search can reach a satisfying point almost in each iteration, which means the second-order optimization method can save much computational effort.


The sub-optimal point with the second-order optimization method in this example is $J(K^*)=159.0686$. It only takes 3.0150 seconds to reach this sub-optimal point. We can see that when the parameters approach closely to the sub-optimal point, this method can achieve second-order convergence, which means that the parameters can converge much faster than the first-order method. The sub-optimal parameter matrix given by the second-order method is
\begin{IEEEeqnarray}{l}
K^*_{(2)}=\left[
\begin{array}{llll}
\phantom+0.3975 & \phantom+1.5925 & \phantom+7.8522 \\
-1.2575 & -3.4823 & -5.0041
\end{array}
\right].
\end{IEEEeqnarray}

\begin{example}
	A third-order system is considered with the following structure,
	\begin{IEEEeqnarray}{rCl}
	\dot x(t) = Ax(t)+Bu(t) \quad y(t)= Cx(t)\quad u(t)=Ky(t),\IEEEeqnarraynumspace 
	\end{IEEEeqnarray}
	where
	\begin{IEEEeqnarray}{l}
	A=\left[
	\begin{array}{lll}
	-4 & \phantom+2 & \phantom+1\\
	\phantom+3 & -2 & \phantom+5\\
	-7 & \phantom+0 & \phantom+3 
	\end{array}
	\right]\quad
	B=\left[
	\begin{array}{ll}
	1 & 0\\
	1 & 0\\
	0 & 1
	\end{array}
	\right]\IEEEnonumber\\
	C=\left[
	\begin{array}{llll}
	0 & 1 & 0\\
	0 & 0 & 1
	\end{array}
	\right].
	\end{IEEEeqnarray}
	
	A decentralized optimal controller $K=\operatorname{diag}\{k_{11},k_{22}\}$ is designed to minimize the cost function given by
	\begin{IEEEeqnarray}{rCl}
	J &=&\displaystyle\int_0^\infty \left(x(t)^TQ x(t)+u(t)^TR u(t)\right) dt,
	\end{IEEEeqnarray}
	where the weighting parameters are chosen as $Q=I,\; R=I$ for demonstrative purposes.
\end{example}

The decentralized linear equality constraints are denoted as
\begin{IEEEeqnarray}{C}
\mathcal A_1^{(1)} K \mathcal B_1^{(1)} = \mathcal C_1^{(1)}, \quad \mathcal A_1^{(2)} K \mathcal B_1^{(2)} = \mathcal C_1^{(2)},
\end{IEEEeqnarray}
where
\begin{IEEEeqnarray}{l}
\mathcal A_1^{(1)}=\left[
\begin{array}{lll}
1 & 0
\end{array}
\right]\quad
\mathcal B_1^{(1)}=\left[
\begin{array}{ll}
0 \\ 1
\end{array} 
\right]\quad
\mathcal C_1^{(1)}=0 \IEEEnonumber\\
\mathcal A_1^{(2)}=\left[
\begin{array}{lll}
0 & 1
\end{array}
\right]\quad
\mathcal B_1^{(2)}=\left[
\begin{array}{ll}
1 \\ 0
\end{array}
\right]\quad
\mathcal C_1^{(2)}=0.
\end{IEEEeqnarray}
By using (\ref{equation:transformation}), we have
\begin{IEEEeqnarray}{l}
\bar {\mathcal A}=\left[
\begin{array}{llll}
0 & 0 & 1 & 0\\
0 & 1 & 0 & 0
\end{array}
\right]\quad
\bar {\mathcal C}=\left[
\begin{array}{ll}
0 \\ 0
\end{array} 
\right].
\end{IEEEeqnarray}

In this example, the stopping criterion is chosen as $\varepsilon = 10^{-9}$ and the initial system state vector is chosen as a random vector with $\mathbb E\left(x_0x_0^T\right)=I$. Both the first-order optimization method and the second-order optimization method utilize the line search method. The parameters for the backtracking line search are chosen as $\alpha = 0.2$ and $\beta = 0.1$. The parameter for the PT-matrix, which will be used in the second-order optimization method, is chosen as $\epsilon = 10^{-6}$. The initial controller gain matrix is chosen as $K_0 = \operatorname{diag}\{-2, -3\}$.

\begin{figure}
	\centering
	\includegraphics[trim=100 50 120 50,width=1\columnwidth]{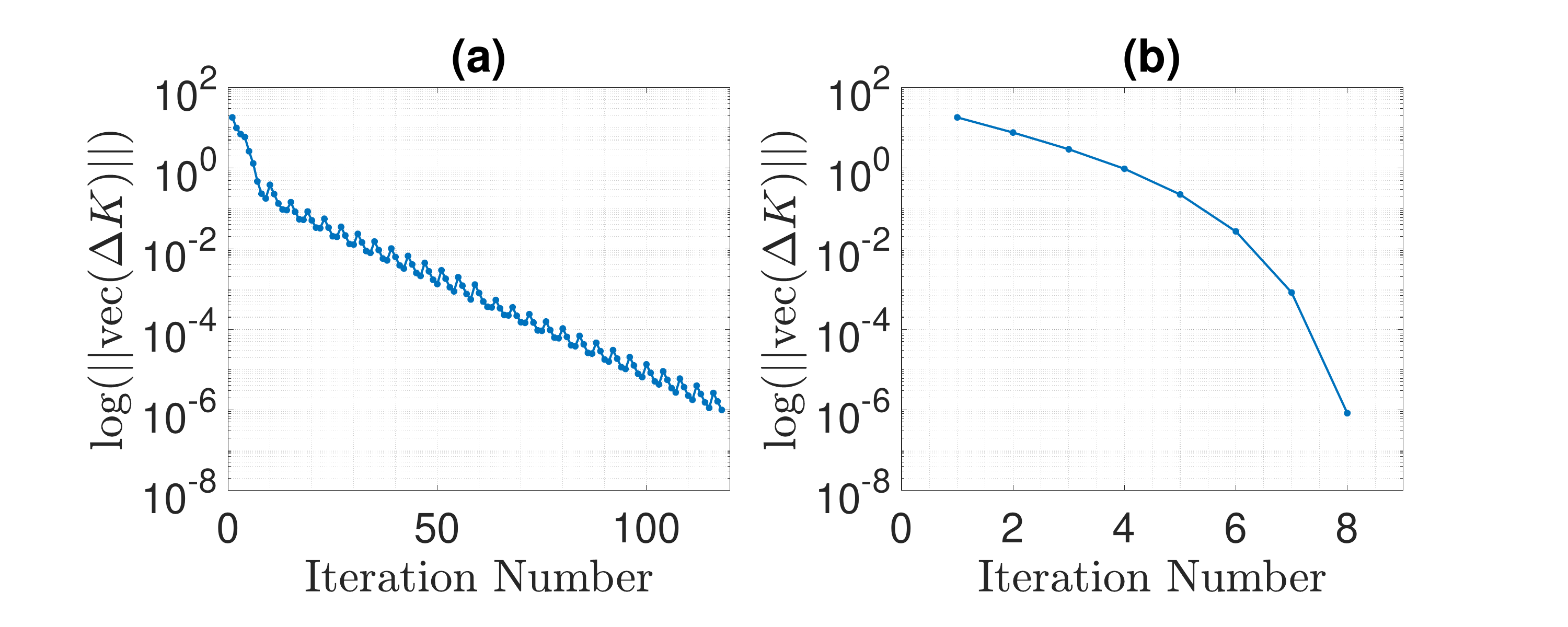}
	\caption{Norm of the gradient in Example 2. (a) First-order method. (b) Proposed method.}
	\label{fig:cost2}
\end{figure}

Fig.~\ref{fig:cost2}(a) shows the norm of the gradient of the cost function with respect to the controller gain matrix during iterations with the first-order optimization method. It can be seen that the first-order optimization method with the gradient projection method takes 118 iterations (totally 209 iterations with the backtracking line search iterations taken into consideration) to satisfy the stopping criterion. It takes 16.9911 seconds to reach the sub-optimal point. The cost function value with respect to the initial controller gain matrix is 22.2010, and after 118 iterations, the value of the cost function decreases to 12.8281. 

Fig.~\ref{fig:cost2}(b) shows the norm of the gradient of the cost function with respect to the controller gain matrix during iterations with the second-order optimization method. It shows that the second-order optimization method with the equality constrained Newton's method only needs 8 iterations (totally 8 iterations with the backtracking line search iterations taken into consideration) to satisfy the stopping criterion. It takes 1.5021 seconds to reach the sub-optimal point. The cost function value with respect to the initial controller gain matrix is 22.2010, and after 8 iterations, the value of the cost function decreases to 12.8281. Compared with the first-order method, the second-order method can achieve a much higher convergence rate. The sub-optimal parameter matrices given by both of the methods are the same with $K^*=\operatorname{diag}\{-1.3211, -6.0723\}$.


\section{Conclusion}\label{section:conclusion}
In this paper, a second-order non-convex optimization method is introduced and proposed
to solve the constrained fixed-structure SOF problem. 
Firstly, an efficient method in the matrix space is proposed to derive the Hessian matrix 
of the cost function with respect to the controller gain matrix. 
Secondly, the PT-inverse method is utilized to cater to the indefiniteness of the Hessian matrix. 
Thirdly, the equality constrained Newton's method is proposed to solve 
the controller optimization problem with the structural constraints. 
Finally, two illustrative examples are given to verify the 
applicability and effectiveness of the
proposed method. 
Comparisons between the first-order method and the second-order method proposed here
show the greatly improved performance of our proposed methodology and algorithm. 
With this proposed algorithm, the SOF LQR problems can certainly 
be solved with the requisite high accuracy and improved effectiveness.

\bibliographystyle{IEEEtran}
\bibliography{IEEEabrv,mybibfile}

\begin{thebibliography}{10}
\providecommand{\url}[1]{#1}
\csname url@samestyle\endcsname
\providecommand{\newblock}{\relax}
\providecommand{\bibinfo}[2]{#2}
\providecommand{\BIBentrySTDinterwordspacing}{\spaceskip=0pt\relax}
\providecommand{\BIBentryALTinterwordstretchfactor}{4}
\providecommand{\BIBentryALTinterwordspacing}{\spaceskip=\fontdimen2\font plus
\BIBentryALTinterwordstretchfactor\fontdimen3\font minus
  \fontdimen4\font\relax}
\providecommand{\BIBforeignlanguage}[2]{{%
\expandafter\ifx\csname l@#1\endcsname\relax
\typeout{** WARNING: IEEEtran.bst: No hyphenation pattern has been}%
\typeout{** loaded for the language `#1'. Using the pattern for}%
\typeout{** the default language instead.}%
\else
\language=\csname l@#1\endcsname
\fi
#2}}
\providecommand{\BIBdecl}{\relax}
\BIBdecl

\bibitem{dolk2016output}
V.~Dolk, D.~P. Borgers, and W.~Heemels, ``Output-based and decentralized
  dynamic event-triggered control with guaranteed {$ L_p$}-gain performance and
  zeno-freeness,'' \emph{IEEE Transactions on Automatic Control}, vol.~62,
  no.~1, pp. 34--49, 2016.

\bibitem{jiang2019optimal}
Y.~Jiang, B.~Kiumarsi, J.~Fan, T.~Chai, J.~Li, and F.~L. Lewis, ``Optimal
  output regulation of linear discrete-time systems with unknown dynamics using
  reinforcement learning,'' \emph{IEEE Transactions on Cybernetics}, vol.~50,
  no.~7, pp. 3147--3156, 2019.

\bibitem{zhang2014distributed}
H.~Zhang, T.~Feng, G.-H. Yang, and H.~Liang, ``Distributed cooperative optimal
  control for multiagent systems on directed graphs: An inverse optimal
  approach,'' \emph{IEEE Transactions on Cybernetics}, vol.~45, no.~7, pp.
  1315--1326, 2014.

\bibitem{wei2015value}
Q.~Wei, D.~Liu, and H.~Lin, ``Value iteration adaptive dynamic programming for
  optimal control of discrete-time nonlinear systems,'' \emph{IEEE Transactions
  on Cybernetics}, vol.~46, no.~3, pp. 840--853, 2015.

\bibitem{wu2018optimal}
G.~Wu, J.~Sun, and J.~Chen, ``Optimal linear quadratic regulator of switched
  systems,'' \emph{IEEE Transactions on Automatic Control}, vol.~64, no.~7, pp.
  2898--2904, 2019.

\bibitem{kanieski2015robust}
J.~M. Kanieski, R.~V. Tambara, H.~Pinheiro, R.~Cardoso, and H.~A. Gruendling,
  ``Robust adaptive controller combined with a linear quadratic regulator based
  on {Kalman} filtering,'' \emph{IEEE Transactions on Automatic Control},
  vol.~61, no.~5, pp. 1373--1378, 2015.

\bibitem{zhang2015linear}
H.~Zhang, L.~Li, J.~Xu, and M.~Fu, ``Linear quadratic regulation and
  stabilization of discrete-time systems with delay and multiplicative noise,''
  \emph{IEEE Transactions on Automatic Control}, vol.~60, no.~10, pp.
  2599--2613, 2015.

\bibitem{levine1970determination}
W.~Levine and M.~Athans, ``On the determination of the optimal constant output
  feedback gains for linear multivariable systems,'' \emph{IEEE Transactions on
  Automatic Control}, vol.~15, no.~1, pp. 44--48, 1970.

\bibitem{ma2019parameter}
J.~Ma, S.-L. Chen, C.~S. Teo, A.~Tay, A.~Al~Mamun, and K.~K. Tan, ``Parameter
  space optimization towards integrated mechatronic design for uncertain
  systems with generalized feedback constraints,'' \emph{Automatica}, vol. 105,
  pp. 149--158, 2019.

\bibitem{sadabadi2016static}
M.~S. Sadabadi and D.~Peaucelle, ``From static output feedback to structured
  robust static output feedback: A survey,'' \emph{Annual Reviews in Control},
  vol.~42, pp. 11--26, 2016.

\bibitem{blondel1997np}
V.~Blondel and J.~N. Tsitsiklis, ``{NP}-hardness of some linear control design
  problems,'' \emph{SIAM Journal on Control and Optimization}, vol.~35, no.~6,
  pp. 2118--2127, 1997.

\bibitem{el1994synthesis}
L.~El~Ghaoui and V.~Balakrishnan, ``Synthesis of fixed-structure controllers
  via numerical optimization,'' in \emph{Proceedings of 33rd IEEE Conference on
  Decision and Control}, 1994, pp. 2678--2683.

\bibitem{lind1994evaluating}
R.~Lind, G.~J. Balas, and A.~Packard, ``Evaluating {D-K} iteration for control
  design,'' in \emph{Proceedings of American Control Conference}, 1994, pp.
  2792--2797.

\bibitem{geromel1994lmi}
J.~C. Geromel, C.~C. de~Souza, and R.~E. Skelton, ``{LMI} numerical solution
  for output feedback stabilization,'' in \emph{Proceedings of American Control
  Conference}, 1994, pp. 40--44.

\bibitem{geromel1998static}
J.~C. {Geromel}, C.~C. {de Souza}, and R.~E. {Skelton}, ``Static output
  feedback controllers: Stability and convexity,'' \emph{IEEE Transactions on
  Automatic Control}, vol.~43, no.~1, pp. 120--125, 1998.

\bibitem{peres1993h}
P.~L.~D. {Peres}, J.~C. {Geromel}, and S.~R. {Souza}, ``{$H_\infty$} robust
  control by static output feedback,'' in \emph{Proceedings of American Control
  Conference}, 1993, pp. 620--621.

\bibitem{el1997cone}
L.~El~Ghaoui, F.~Oustry, and M.~AitRami, ``A cone complementarity linearization
  algorithm for static output-feedback and related problems,'' \emph{IEEE
  Transactions on Automatic Control}, vol.~42, no.~8, pp. 1171--1176, 1997.

\bibitem{ma2017integrated}
J.~Ma, S.-L. Chen, N.~Kamaldin, C.~S. Teo, A.~Tay, A.~Al~Mamun, and K.~K. Tan,
  ``Integrated mechatronic design in the flexure-linked dual-drive gantry by
  constrained linear-quadratic optimization,'' \emph{IEEE Transactions on
  Industrial Electronics}, vol.~65, no.~3, pp. 2408--2418, 2017.

\bibitem{chanekar2017optimal}
P.~V. Chanekar, N.~Chopra, and S.~Azarm, ``Optimal structured static output
  feedback design using generalized benders decomposition,'' in
  \emph{Proceedings of 56th IEEE Annual Conference on Decision and Control},
  2017, pp. 4819--4824.

\bibitem{tassa2012synthesis}
Y.~Tassa, T.~Erez, and E.~Todorov, ``Synthesis and stabilization of complex
  behaviors through online trajectory optimization,'' in \emph{Proceedings of
  IEEE Conference on Intelligent Robots and Systems}, 2012, pp. 4906--4913.

\bibitem{lin2013design}
F.~Lin, M.~Fardad, and M.~R. Jovanovi{\'c}, ``Design of optimal sparse feedback
  gains via the alternating direction method of multipliers,'' \emph{IEEE
  Transactions on Automatic Control}, vol.~58, no.~9, pp. 2426--2431, 2013.

\bibitem{geromel1994decentralized}
J.~C. Geromel, J.~Bernussou, and P.~L.~D. Peres, ``Decentralized control
  through parameter space optimization,'' \emph{Automatica}, vol.~30, no.~10,
  pp. 1565--1578, 1994.

\bibitem{vidyasagar2002nonlinear}
M.~Vidyasagar, \emph{Nonlinear Systems Analysis}.\hskip 1em plus 0.5em minus
  0.4em\relax Englewood Cliffs, NJ, USA: Prentice-Hall, 1993.

\bibitem{crusius1999sufficient}
C.~A. Crusius and A.~Trofino, ``Sufficient {LMI} conditions for output feedback
  control problems,'' \emph{IEEE Transactions on Automatic Control}, vol.~44,
  no.~5, pp. 1053--1057, 1999.

\bibitem{geromel1979algorithm}
J.~C. Geromel and J.~Bernussou, ``An algorithm for optimal decentralized
  regulation of linear quadratic interconnected systems,'' \emph{Automatica},
  vol.~15, no.~4, pp. 489--491, 1979.

\bibitem{dauphin2014identifying}
Y.~N. Dauphin, R.~Pascanu, C.~Gulcehre, K.~Cho, S.~Ganguli, and Y.~Bengio,
  ``Identifying and attacking the saddle point problem in high-dimensional
  non-convex optimization,'' in \emph{Advances in Neural Information Processing
  Systems}, 2014, pp. 2933--2941.

\bibitem{nocedal2006numerical}
J.~Nocedal and S.~Wright, \emph{Numerical Optimization}.\hskip 1em plus 0.5em
  minus 0.4em\relax New York, NY, USA: Springer, 2006.

\bibitem{paternain2019newton}
S.~Paternain, A.~Mokhtari, and A.~Ribeiro, ``A {Newton}-based method for
  nonconvex optimization with fast evasion of saddle points,'' \emph{SIAM
  Journal on Optimization}, vol.~29, no.~1, pp. 343--368, 2019.

\bibitem{dong2013robust}
J.~Dong and G.-H. Yang, ``Robust static output feedback control synthesis for
  linear continuous systems with polytopic uncertainties,'' \emph{Automatica},
  vol.~49, no.~6, pp. 1821--1829, 2013.

\bibitem{bartels1972solution}
R.~H. Bartels and G.~W. Stewart, ``Solution of the matrix equation {AX + XB =
  C},'' \emph{Communications of the ACM}, vol.~15, no.~9, pp. 820--826, 1972.

\bibitem{coppersmith1982asymptotic}
D.~Coppersmith and S.~Winograd, ``On the asymptotic complexity of matrix
  multiplication,'' \emph{SIAM Journal on Computing}, vol.~11, no.~3, pp.
  472--492, 1982.

\bibitem{choi1974computation}
S.~Choi and H.~Sirisena, ``Computation of optimal output feedback gains for
  linear multivariable systems,'' \emph{IEEE Transactions on Automatic
  Control}, vol.~19, no.~3, pp. 257--258, 1974.

\end{thebibliography}

\end{document}